\documentclass[11pt]{amsart}
\usepackage[margin=.9in]{geometry}
\usepackage{latexsym}
\usepackage{amsfonts}
\usepackage{amsmath}
\usepackage{amssymb}
\usepackage{amsthm}
\usepackage{enumerate}
\usepackage[hang,flushmargin]{footmisc}
\usepackage{caption}
\usepackage{tabu}
\usepackage{mathrsfs}
\usepackage{amsaddr}

\usepackage{graphicx}
\usepackage{epstopdf}
\usepackage{epsfig}
 
\usepackage{bm} 

\newtheorem{theorem}{Theorem}[section]
\newtheorem{lemma}{Lemma}[section]

\newtheorem{corollary}{Corollary}[section]
\newtheorem{fact}{Fact}[section]

\theoremstyle{definition}
\newtheorem{definition}{Definition}[section]
\newtheorem{remark}{Remark}[section]
\newtheorem{example}{Example}[section]
\newtheorem{question*}{Question}

\begin{document}
\title{Preimages under the Stack-Sorting Algorithm}
\author{Colin Defant}
\address{University of Florida \\ 1400 Stadium Rd. \\ Gainesville, FL 32611 United States}
\email{cdefant@ufl.edu}

\maketitle

\begin{abstract}
We use a method for determining the number of preimages of any permutation under the stack-sorting map in order to obtain recursive upper bounds for the numbers $W_t(n)$ and $W_t(n,k)$ of $t$-stack sortable permutations of length $n$ and $t$-stack sortable permutations of length $n$ with exactly $k$ descents. From these bounds, we are able to significantly improve the best known upper bounds for $\displaystyle{\lim_{n\to\infty}\sqrt[n]{W_t(n)}}$ when $t=3$ and $t=4$.      
\end{abstract}
\vspace{.3cm}
\hrule



\section{Introduction} 

Throughout this paper, we let $S_n$ denote the permutations of the set $[n]=\{1,2,\ldots,n\}$. Recall that a \emph{descent} of a permutation $\sigma=\sigma_1\sigma_2\cdots\sigma_n$ is an index $i$ such that $\sigma_i>\sigma_{i+1}$ (we do not include $n$ as a descent). If $i$ is a descent of $\sigma$, then the entry $\sigma_i$ is called a \emph{decent top} of $\sigma$. If $\sigma$ has $k$ descents, then we may write $\sigma$ as the concatenation of $k+1$ increasing subsequences, which are called the \emph{ascending runs} of $\sigma$. 

In his 1990 Ph.D. thesis, Julian West \cite{West90} studied a function $s$ that transforms permutations into permutations through the use of a vertical stack. We call the function $s$ the \emph{stack-sorting map}. Given an input permutation $\pi=\pi_1\pi_2\cdots\pi_n$, the permutation $s(\pi)$ is computed by the following algorithm. At any point in time during the algorithm, if the next entry in the input permutation is larger than the entry at the top of the stack or if the stack is empty, the next entry in the input permutation is placed at the top of the stack. Otherwise, the entry at the top of the stack is annexed to the end of the growing output permutation. For example, $s(35214)=31245$. 

The following observation due to West provides an alternative recursive means of defining the stack-sorting map. 
\begin{fact} \label{Fact1} 
Let $\pi$ be a permutation of positive integers with largest entry $n$, and write $\pi=LnR$, where $L$ (respectively, $R$) is the (possibly empty) substring of $\pi$ to the left (respectively, right) of the entry $n$. Then $s(\pi)=s(L)s(R)n$.
\end{fact}

As the name suggests, the purpose of the stack-sorting map is to sort a permutation $\pi$ into the identity permutation $123\cdots n$. We have seen that $s(35214)=31245$, so the image $s(\pi)$ of a permutation $\pi$ is not always the identity permutation. Nevertheless, it follows easily from Fact \ref{Fact1} that $s^{n-1}(\pi)=123\cdots n$ for any $\pi\in S_n$. Therefore, we can sort any permutation if we are allowed to use the stack iteratively. This leads to the following definition. 
\begin{definition} 
Let $t$ be a positive integer. A permutation $\pi\in S_n$ is called \emph{$t$-stack sortable} if $s^t(\pi)=123\cdots n$. A permutation that  is $1$-stack sortable is simply called \emph{sortable}. Let $W_t(n)$ denote the number of $t$-stack sortable permutations of length $n$. Let $W_t(n,k)$ denote the number of $t$-stack sortable permutations of length $n$ with exactly $k$ descents. 
\end{definition}

It is well-known that $W_1(n)$ is equal to  $C_n$, the $n^\text{th}$ Catalan number. In his dissertation, West conjectured that 
\begin{equation}\label{Eq24}
W_2(n)=\frac{2}{(n+1)(2n+1)}{3n\choose n}.
\end{equation} Doron Zeilberger 
provided the first proof of this result two years later \cite{Zeilberger92}. Subsequent combinatorial proofs \cite{Dulucq2,Goulden} of \eqref{Eq24} provided bijections between $2$-stack sortable permutations of length $n$ and rooted nonseparable planar maps with $n+1$ edges, the latter of which Tutte had already enumerated \cite{Tutte}. In 1997, Cori, Jacquard, and Schaeffer \cite{Cori} found a bijection between $2$-stack sortable permutations of length $n$ and the class of $\beta(1,0)$-trees. In \cite{Bona}, B\'ona briefly calls the reader's attention to lattice paths that use the steps $(1,1)$, $(0,-1)$, and $(-1,0)$ and never leave the first quadrant. He notes that the number of such paths using $3n$ steps that start and end at $(0,0)$ is equal to $2^{2n-1}W_2(n)$ (see \cite{Bousquet02}), so he naturally inquires about the possibility of using this result to obtain a simple combinatorial proof of \eqref{Eq24}.

For any given $t\geq 3$, obtaining an explicit formula for $W_t(n)$ seems to be a highly 
formidable task. Indeed, the best known upper bound \cite[Theorem 3.4]{Bona03} is given by $W_t(n)\leq (t+1)^{2n}$. B\'ona \cite{Bona} has conjectured the much stronger upper bound $W_t(n)\leq{(t+1)n\choose n}$. Recently, \'Ulfarsson has given a description of $3$-stack sortable permutations in terms of the avoidance of mesh patterns and new types of patterns that he calls ``decorated patterns" \cite{Ulfarsson}. However, his description does not immediately lend itself to any means of enumerating $3$-stack sortable permutations. We refer the reader to \cite{Bona} and \cite{Bona03} for more thorough treatments of the history of stack-sorting algorithms. 

West defined the \emph{fertility} $F(\pi)$ of a permutation $\pi$ to be the number of preimages of $\pi$ under the stack-sorting map. Bousquet-M\'elou \cite{Bousquet00} then defined a \emph{sorted} permutation to be a permutation whose fertility is positive. Observe that the sortable permutations of length $n$ are precisely the preimages of the identity permutation $123\cdots n$, so the fertitily of $123\cdots n$ is $C_n$. The $(t+1)$-stack sortable permutations are precisely the preimages of $t$-stack sortable permutations. Consequently, $W_{t+1}(n)$ is the sum of the fertilities of the $t$-stack sortable permutations of length $n$. These observations form the central motivation for this paper. 

In \cite{Defant}, the author used a geometric construction of objects called \emph{valid hook configurations} to give an expression for the fertility of any given permutations. The first goal of this paper is to translate that expression into a similar expression that relies on a certain set of compositions that we call \emph{valid compositions}. In fact, for any permutation $\pi$ and nonnegative integer $m$, we will be able to determine, in terms of valid compositions, the number $F(\pi,m)$ of permutations $\sigma$ such that $\sigma$ has $m$ descents and $s(\sigma)=\pi$. We can also determine the number of permutations $\sigma$ such that $\sigma$ has $m$ valleys and $s(\sigma)=\pi$ (a valley of $\sigma=\sigma_1\sigma_2\cdots\sigma_n$ is an index $i\in\{1,2,\ldots,n\}$ such that $\sigma_i<\min\{\sigma_{i-1},\sigma_{i+1}\}$, where we make the convention $\sigma_0=\sigma_{n+1}=\infty$).   

The paper is organized as follows. Section 2 provides a discussion of decreasing binary plane trees and valid hooks configurations, introducing terminology that we will need in subsequent sections. In Section 3, we define the set of valid compositions of a permutation and show how to express $F(\pi)$ and $F(\pi,m)$ as a sum over the valid compositions of $\pi$. In the final section of the paper, we give a short combinatorial proof of an identity involving generalized Narayana numbers. This identity allows us to prove recursive upper bounds for the numbers $W_t(n)$ and $W_t(n,k)$. We then prove that \[\lim_{n\to\infty}\sqrt[n]{W_3(n)}\leq 12.53296\] (improving upon the best known upper bound of $16$) and \[\lim_{n\to\infty}\sqrt[n]{W_4(n)}\leq 21.97225\] (improving upon the best known upper bound of $25$).  

\section{Trees and Hooks}
A decreasing binary plane tree is a rooted binary plane tree whose nodes are labeled with distinct positive integers such that the label of any non-root node is smaller than the label of its parent. 

To read a decreasing binary plane tree in \emph{postorder}, we first read the left subtree of the root, then the right subtree of the root, and finally the root. Each subtree is itself read in postorder. We let $P(\tau)$ denote the postorder reading of a decreasing binary plane tree $\tau$. 

Suppose $Y$ is a collection of decreasing (not necessarily binary) plane trees and $\pi\in S_n$. In \cite{Defant}, the author asked for an expression for the number of trees in $Y$ with postorder $\pi$. Using a geometric construction of objects called valid hook configurations, he gave such an expression for certain collections $Y$. There is a simple bijection \cite{Bona} between the decreasing binary plane trees with postorder $\pi$ and the preimages of $\pi$ under $s$. Therefore, the author was able to apply this method to gain information about the preimages of $\pi$ under $s$. We will restate the relevant definitions and results from that paper for easy reference.  

\begin{figure}[t]
\begin{center} 
\includegraphics[height=2cm]{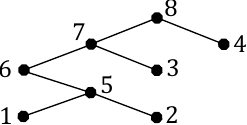}
\end{center}
\captionof{figure}{A decreasing binary plane tree on [8]. The postorder reading of this tree is $12563748$.} \label{Fig3}
\end{figure}

Given a permutation $\pi=\pi_1\pi_2\ldots\pi_n\in S_n$, we obtain a standard diagram for $\pi$ by plotting the points of the form $(i,\pi_i)$ in the plane. A \emph{hook} in this diagram is the union of two line segments. One is a vertical line segment connecting a point $(i,\pi_i)$ to a point $(i,\pi_j)$, where $i<j$ and $\pi_i<\pi_j$. The second is a horizontal line segment connecting the points $(i,\pi_j)$ and $(j,\pi_j)$. One can think of drawing a hook by starting at the point $(i,\pi_i)$, moving upward, and then turning right to meet with the point $(j,\pi_j)$. The point $(i,\pi_i)$ is the \emph{southwest endpoint} of the hook while $(j,\pi_j)$ is the \emph{northeast endpoint} of the hook. We let ${}_eH$ and $H^e$ denote the southwest and northeast endpoints, respectively, of the hook $H$. 

\begin{definition}\label{Def5}
Let $\pi\in S_n$. We say that an $m$-tuple $\mathscr H=(H_1,H_2,\ldots,H_m)$ is a \emph{valid hook configuration of} $\pi$ if $H_1,H_2,\ldots,H_m$ are hooks in the diagram of $\pi$ that satisfy the following properties. 
\begin{enumerate}[(a)]
\item If ${}_eH_\ell=(i_\ell,\pi_{i_\ell})$ for each $\ell\in[m]$, then $i_1<i_2<\cdots<i_m$. 
\item If $i$ is a descent of $\pi$, then $(i,\pi_i)={}_eH_\ell$ for some $\ell\in[m]$. 
\item If $(j,\pi_j)=H_\ell^e$ for some $\ell\in[m]$, then there exist $\ell',\ell''\in[m]$ such that the $x$-coordinate of ${}_eH_{\ell'}$ is a descent of $\pi$, $(j-1,\pi_{j-1})={}_eH_{\ell''}$, and $H_{\ell'}^e=H_{\ell''}^e=(j,\pi_j)$. 
\item If $\ell,\ell'\in [m]$, ${}_eH_\ell=(i,\pi_i)$, $H_\ell^e=(j,\pi_j)$, ${}_eH_{\ell'}=(i',\pi_{i'})$, $H_{\ell'}^e=(j',\pi_{j'})$, $\pi_j\leq\pi_{j'}$, and $\left|[i,j]\cap[i',j']\right|>1$, then $[i,j]\subseteq[i',j']$. 
\end{enumerate}
Let $SW(\mathscr H)=\{{}_eH_1,{}_eH_2,\ldots,{}_eH_m\}$ and $NE(\mathscr H)=\{H_1^e,H_2^e,\ldots,H_m^e\}$. Let $\mathcal H(\pi)$ denote the set of valid hook configurations of $\pi$.   
\end{definition}

Notice that if $\mathscr H\in\mathcal H(\pi)$, then each point in $NE(\mathscr H)$ is the northeast endpoint of at least two hooks: one with southwest endpoint $(j-1,\pi_{j-1})$ and one with southwest endpoint $(i,\pi_i)$ for some descent $i$ of $\pi$. Hence, we have the following definition.  

\begin{definition}\label{Def6}
For $\pi\in S_n$, let $\mathcal H_{\{0,2\}}(\pi)$ be the set of valid hook configurations $\mathscr H\in\mathcal H(\pi)$ in which each point in $NE(\mathscr H)$ is the northeast endpoint of \emph{exactly} two hooks.
\end{definition}

There are some immediate consequences of the criteria in the above definitions that are useful to keep in mind. First, the only way that two hooks can intersect in exactly one point is if that point is the northeast endpoint of one of the hooks and the southwest endpoint of the other. Also, no entry in the diagram can lie above a hook. More formally, if $H_\ell$ is a hook with ${}_eH_\ell=(i,\pi_i)$ and $H_\ell^e=(j,\pi_j)$, then $\pi_k<\pi_j$ for all $k\in\{i+1,i+2,\ldots,j-1\}$. Indeed, suppose instead that $\pi_k>\pi_j$ for some $k\in\{i+1,i+2,\ldots,j-1\}$. Then $\pi$ must have a descent $i'\in\{k,k+1,\ldots,j-1\}$ such that $\pi_{i'}>\pi_j$. According to criterion (b) in the above definition, $(i',\pi_{i'})={}_eH_{\ell'}$ for some hook $H_{\ell'}$. Let $H_{\ell'}^e=(j',\pi_{j'})$. Since $\pi_j<\pi_{i'}<\pi_{j'}$ and $[i',i'+1]\subseteq[i,j]\cap [i',j']$, condition (d) in the above definition states that we must have $[i,j]\subseteq [i',j']$. However, this is impossible because $i<i'$. 

After drawing a diagram of a permutation $\pi$ with a valid hook configuration $\mathscr H=(H_1,H_2,\ldots,$ $H_m)\in\mathcal H(\pi)$, we can color the diagram with $m+1$ colors $c_0,c_1,\ldots,c_m$ as follows. First, if ${}_eH_\ell=(i,\pi_i)$ and $H_\ell^e=(j,\pi_j)$, then we refer to the line segment connecting the points $(i+1/2,\pi_j)$ and $(j,\pi_j)$ as the \emph{top part} of the hook $H_\ell$. Assign $H_\ell$ the color $c_\ell$ for each $\ell\in[m]$. Color each point $(k,\pi_k)$ as follows. Start at $(k,\pi_k)$, and move directly upward until hitting the top part of a hook. Color $(k,\pi_k)$ the same color as the hook that you hit. If you hit multiple hooks at once, use the color of the hook that was hit whose southwest endpoint if farthest to the right. If you do not hit the top part of any hook, give $(k,\pi_k)$ the color $c_0$. Note that if $(k,\pi_k)={}_eH_\ell$ for some $\ell\in m$, then we ignore the hook $H_\ell$ while moving upward from $(k,\pi_k)$ to find the hook that is to lend its color to $(k,\pi_k)$. Moreover, if $(k,\pi_k)\in NE(\mathscr H)$, then we give $(k,\pi_k)$ the color $c_r$, where $r$ is the largest element of $[m]$ such that $(k,\pi_k)=H_r^e$. 

\begin{example}\label{Exam2}
Figure \ref{Fig2} depicts the colored diagram obtained from a valid hook configuration $\mathscr H$ of the permutation $\pi=2.7.3.5.9.10.11.4.8.1.6.12.13.14.15.16$. In this example, $NE(\mathscr H)=\{(7,11), (13,13),$ $(15,15)\}$. This valid hook configuration is an element of $\mathcal H_{\{0,2\}}(\pi)$ because each of the points in $NE(\mathscr H)$ is the northeast endpoint of exactly two hooks.    

\begin{figure}[t]
\begin{center} 
\includegraphics[height=6.0cm]{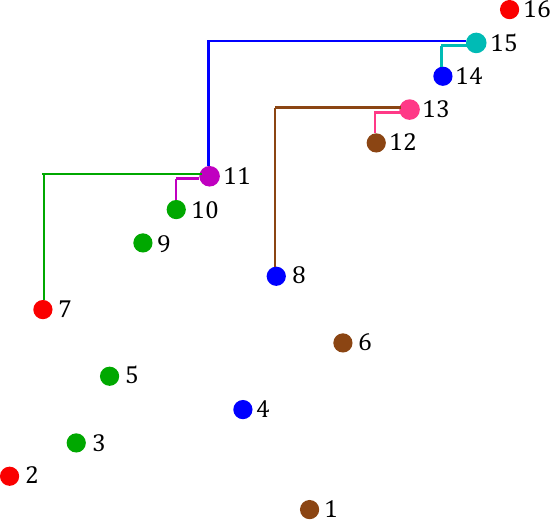}
\end{center}
\captionof{figure}{The colored diagram arising from a valid hook configuration $\mathscr H\in\mathcal H_{\{0,2\}}(\pi)$.}\label{Fig2}
\end{figure}
\end{example}

\begin{definition}\label{Def7}
If $\pi\in S_n$, then each valid hook configuration $\mathscr H=(H_1,H_2,\ldots,H_m)\in\mathcal H(\pi)$ partitions $[n]$ into color classes. Let $Q_t(\mathscr H)$ be the set of entries $\pi_\ell$ such that  $(\ell,\pi_\ell)$ is assigned the color $c_t$. Let $q_t(\mathscr H)=\vert Q_t(\mathscr H)\vert$ so that $(q_0(\mathscr H),q_1(\mathscr H),\ldots,q_m(\mathscr H))$ is a composition of $n$ into $m+1$ parts. Let $\vert \mathscr H\vert=m$ denote the number of hooks in the valid hook configuration $\mathscr H$. 
\end{definition}

\begin{remark}\label{Rem6}
Suppose $\mathscr H\in\mathcal H(\pi)$ and $(j,\pi_j)\in NE(\mathscr H)$. If $(j,\pi_j)$ is given the color $c_t$ in the colored diagram of $\pi$ induced by $\mathscr H$, then $Q_t(\mathscr H)=\{\pi_j\}$ and $q_t(\mathscr H)=1$. Indeed, the hook colored $c_t$ is the hook with southwest endpoint $(j-1,\pi_{j-1})$ and northeast endpoint $(j,\pi_j)$. 
\end{remark}

If $\mathscr H$ is a valid hook configuration of $\pi$, let $\Theta(\mathscr H)$ be the set of all $i\in\{0,1,\ldots,\vert\mathscr H\vert\}$ such that the color $c_i$ is not used in the colored diagram induced by $\mathscr H$ to color a point in $NE(\mathscr H)$. We let $\widehat{\vert\mathscr H\vert}=\vert\Theta(\mathscr H)\vert-1=\vert \mathscr H\vert-\vert NE(\mathscr H)\vert$. For example, the colors $c_2$, $c_5$, and $c_6$ (purple, pink, and teal) are used in Example \ref{Exam2} to colors the points in $NE(\mathscr H)$. Therefore, in that example, we have $\Theta(\mathscr H)=\{0,1,3,4\}$ and $\widehat{\vert\mathscr H\vert}=3$. 
If $\Theta(\mathscr H)=\left\{i_0,i_1,\ldots,i_{\widehat{\vert\mathscr H\vert}}\right\}$, where $i_0<i_1<\cdots<i_{\widehat{\vert\mathscr H\vert}}$, then we let $\widehat q_t(\mathscr H)=q_{i_t}(\mathscr H)$. In other words, the tuple $\left(\widehat q_0(\mathscr H),\widehat q_1(\mathscr H),\ldots,\widehat q_{\widehat{\vert \mathscr H\vert}}(\mathscr H)\right)$ is obtained by starting with the tuple $\left(q_0(\mathscr H),q_1(\mathscr H),\ldots,q_{\vert\mathscr H\vert}(\mathscr H)\right)$
and removing all of the coordinates $q_i(\mathscr H)$ such that the color $c_i$ is used to color a point in $NE(\mathscr H)$. 

\begin{remark}\label{Rem7}
If $\pi\in S_n$ has exactly $k$ descents and $\mathscr H\in\mathcal H_{\{0,2\}}(\pi)$, then  there are exactly $k$ elements of $NE(\mathscr H)$, and there are $2k$ hooks in $\mathscr H$. Therefore, $\vert\widehat{\mathscr H}\vert=\vert\Theta(\mathscr H)\vert-1=\vert \mathscr H\vert-\vert NE(\mathscr H)\vert=2k-k=k$.  
\end{remark}

Let $C_i=\frac{1}{i+1}{2i\choose i}$ be the $i^\text{th}$ Catalan number. Let $N(i,j)=\frac{1}{i}{i\choose j}{i\choose j-1}$ be a Narayana number. The following lemmas are Theorem 5.1, Theorem 5.2, and Corollary 5.1 in \cite{Defant}.   

\begin{lemma}\label{Lem15}
If $\pi\in S_n$ is a permutation with exactly $k$ descents, then the number of permutations $\sigma\in S_n$ such that $s(\sigma)=\pi$ is given by \[F(\pi)=\sum_{\mathscr H\in\mathcal H_{\{0,2\}}(\pi)}\prod_{t=0}^kC_{\widehat{q}_t(\mathscr H)}.\]
\end{lemma} 

\begin{lemma}\label{Lem16}
If $\pi\in S_n$ has exactly $k$ descents, then the number of permutations $\sigma\in S_n$ that have exactly $m$ descents and satisfy $s(\sigma)=\pi$ is given by \[F(\pi,m)=\sum_{\substack{\mathscr H\in\mathcal H_{\{0,2\}}(\pi)\\ j_0+j_1+\cdots+j_k=m-k}}\prod_{t=0}^kN(\widehat{q}_t(\mathscr H),j_t+1),\] where the numbers $j_0,j_1,\ldots,j_k$ in the sum are nonnegative integers that sum to $m-k$. 
\end{lemma}

\begin{lemma}\label{Lem17}
If $\pi\in S_n$ has exactly $k$ descents, then the number of permutations $\sigma\in S_n$ that have exactly $m$ valleys and satisfy $s(\sigma)=\pi$ is given by
\[2^{n-2m+1}\sum_{\substack{\mathscr H\in\mathcal H_{\{0,2\}}(\pi)\\j_0+j_1+\cdots+j_k=m}}\prod_{t=0}^k {\widehat q_t(\mathscr H)-1\choose 2j_t-2}C_{j_t-1},\] where the numbers $j_0,j_1,\ldots,j_k$ in the sum are nonnegative integers that sum to $m$.  
\end{lemma}    

\section{Valid Compositions} 
Throughout this section, let $\pi=\pi_1\pi_2\cdots\pi_n\in S_n$ be a permutation with exactly $k$ descents. Let $d_1,d_2,\ldots,d_k$ be the descents of $\pi$ in increasing order. We convene to let $d_0=0$ and $d_{k+1}=n$.

Any valid hook configuration in $\mathcal H_{\{0,2\}}(\pi)$ is uniquely determined by the $k$-tuple $(b_1,b_2,\ldots,b_k)$, where $b_i$ is the index such that $(b_i,\pi_{b_i})$ is the northeast endpoint of the hook whose southwest endpoint is $(d_i,\pi_{d_i})$. We call $(b_1,b_2,\ldots,b_k)$ the $k$-tuple corresponding to the valid hook configuration. For an explicit example, consider the permutation and the valid hook configuration in Example \ref{Exam2}. This permutation has $k=3$ descents, which are $d_1=2$, $d_2=7$, and $d_3=9$. The point $(d_1,\pi_{d_1})=(2,7)$ is the southwest endpoint of a hook whose northeast endpoint is $(7,11)$, so $b_1=7$. Similarly, $b_2=15$ and $b_3=13$. 

Let $\text{Comp}_a(b)$ denote the set of all compositions of $b$ into $a$ parts. It follows from Remarks \ref{Rem6} and \ref{Rem7} that there is a map $\varphi\colon\mathcal H_{\{0,2\}}(\pi)\to \text{Comp}_{k+1}(n-k)$ given by 
\begin{equation}\label{Eq12}
\varphi(\mathscr H)=(\widehat q_0(\mathscr H),\widehat q_1(\mathscr H),\ldots,\widehat q_k(\mathscr H)).
\end{equation}
The injectivity of $\varphi$, which we shall prove in the following lemma, will be the cornerstone for our subsequent proofs. 

\begin{lemma}\label{Lem1}
The map $\varphi\colon\mathcal H_{\{0,2\}}(\pi)\to \text{Comp}_{k+1}(n-k)$ defined by $\varphi(\mathscr H)=(\widehat q_0(\mathscr H),\ldots,\widehat q_k(\mathscr H))$ is injective. 
\end{lemma}
\begin{proof}
Suppose we are given the composition $\varphi(\mathscr H)=(\widehat q_0(\mathscr H),\widehat q_1(\mathscr H),\ldots,\widehat q_k(\mathscr H))$. Let $(b_1,b_2,\ldots,$ $b_k)$ be the $k$-tuple corresponding to $\mathscr H$. In $\mathscr H$, let $\widehat H_i$ be the hook whose southwest endpoint is $(d_i,\pi_{d_i})$ and whose northeast endpoint is $(b_i,\pi_{b_i})$. We can determine the value of $b_k$ by noting that $b_k=d_k+\widehat q_k(\mathscr H)+1$. Indeed, $\widehat H_k$ must lie above exactly $\widehat q_k(\mathscr H)$ points to the right of $(d_k,\pi_{d_k})$. Similarly, $\widehat H_{k-1}$ must lie above $\widehat q_{k-1}(\mathscr H)$ points that do not lie below $\widehat H_k$ and are not $(b_k,\pi_{b_k})$. Therefore, $b_{k-1}$ is uniquely determined. In general, $b_\ell$ is determined once we know $b_{\ell+1},b_{\ell+2},\ldots,b_k$.  Consequently, $(b_1,b_2,\ldots,b_k)$ is uniquely determined by $\varphi(\mathscr H)$. It follows that $\mathscr H$ is uniquely determined.     
\end{proof}

\begin{definition}\label{Def1}
Let $V(\pi)=\varphi(\mathcal H_{\{0,2\}}(\pi))$. We say that a composition in $V(\pi)$ is a \emph{valid composition of $\pi$}.
\end{definition}

In order to produce a means of determining whether or not a given permutation is sorted, Bousquet-M\'elou introduced the notion of a \emph{canonical tree} of a permutation \cite{Bousquet00}. She then showed that a permutation is sorted if and only if it has a canonical tree. Each sorted permutation has a unique canonical tree, and any two permutations whose canonical trees have the same shape have the same fertility. Because of this last observation, Bousquet-M\'elou mentioned that it would be interesting to find a method for determining the fertility of a permutation from the shape of its canonical tree. This is precisely what we shall do. However, we will translate the notion of a canonical tree into the language of valid hook configurations by defining a \emph{canonical valid hook configuration} $\mathscr H^*=(H_1^*,H_2^*,\ldots,H_{2k}^*)\in\mathcal H_{\{0,2\}}(\pi)$.  

The construction of $\mathscr H^*$ is quite simple. The idea is essentially to choose each northeast endpoint ``minimally," although we need to be a bit careful when doing so. Implicit in what follows is the assumption that $\mathcal H_{\{0,2\}}(\pi)$ is nonempty (so each choice of a northeast endpoint will be possible). We are going to construct $\mathscr H^*$ by building its corresponding $k$-tuple, which we denote $\bm{b}^*=(b_1^*,b_2^*,\ldots,b_k^*)$.  

We will define the entries of $\bm{b}^*$ in the order $b_k^*,b_{k-1}^*,\ldots,b_1^*$. Suppose we have already chosen $b_k^*,b_{k-1}^*,\ldots,b_{\ell+1}^*$ and that we now need to choose $b_\ell^*$. For each $i\in\{\ell+1,\ell+2,\ldots,k\}$, let $H_i^*$ be the hook with southwest endpoint $(d_i,\pi_{d_i})$ and northeast endpoint $(b_i^*,\pi_{b_i^*})$. Let $Z_\ell^*$ be the set of entries $\pi_y$ to the right of $\pi_{d_\ell}$ in $\pi$ that are not in the set $\{\pi_{b_{\ell+1}^*},\pi_{b_{\ell+2}^*},\ldots,\pi_{b_k^*}\}$ such that $(y,\pi_y)$ does not lie below any of the hooks $H_{\ell+1}^*,H_{\ell+2}^*,\ldots,H_k^*$. Of the entries in $Z_\ell^*$, let $\pi_{b_\ell^*}$ be the smallest one that is greater than $\pi_{d_\ell}$. In particular, $\pi_{b_k^*}$ is the smallest entry that is greater than $\pi_{d_k}$ and that appears to the right of $\pi_{d_k}$ in $\pi$. We refer to the entries $b_1^*,b_2^*,\ldots,b_k^*$ in $\bm{b}^*$ as \emph{canonical northeast endpoints}. Figure \ref{Fig6} shows the canonical valid hook configuration of the permutation $\pi$ from Example \ref{Exam2}. 

\begin{remark}\label{Rem8}
It is possible to massage the ideas in \cite{Defant} in order to obtain a one-to-one correspondence between decreasing binary plane trees with postorder $\pi$ and pairs $(\mathscr H,\mathscr T)$, where $\mathscr H\in\mathcal H_{\{0,2\}}(\pi)$ and $\mathscr T$ is a certain tuple of decreasing binary plane trees. If $\mathscr H=\mathscr H^*$ is the canonical valid hook configuration of $\pi$ and the trees in the tuple $\mathscr T$ are chosen so that every edge in every tree is a right edge, then the corresponding tree with postorder $\pi$ will be the canonical tree that Bousquet-M\'elou introduced.  
\end{remark}

\begin{figure}[t]
\begin{center} 
\includegraphics[height=6.0cm]{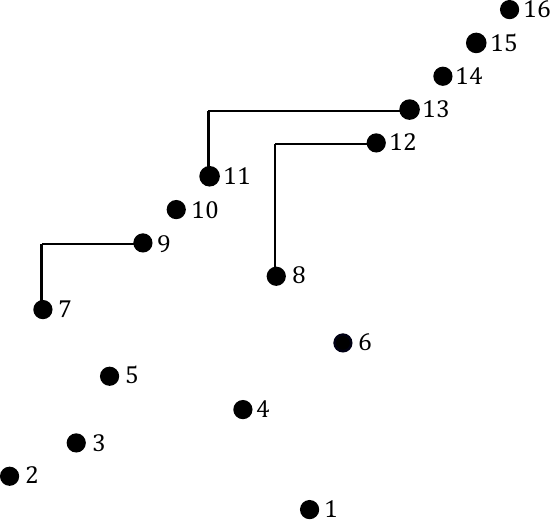}
\end{center}
\captionof{figure}{The permutation $\pi=2.7.3.5.9.10.11.4.8.1.6.12.13.14.15.
16$ from Example \ref{Exam2} has a canonical valid hook configuration with corresponding $k$-tuple $\bm{b}^*=(9,13,12)$. } \label{Fig6}
\end{figure}

With $\varphi$ as in \eqref{Eq12}, let $\varphi(\mathscr H^*)=(\mu_0,\mu_1,\ldots,\mu_k)$. Suppose each entry $\pi_{b_i^*}$ is in the $e_i^\text{th}$ ascending run of $\pi$. In other words, we define $e_i$ for $i\in\{1,\ldots,k\}$ by requiring that $d_{e_i-1}<b_i^*\leq d_{e_i}$ (again, $d_j$ is the $j^{\text{th}}$ descent of $\pi$ while we use the conventions $d_0=0$ and $d_{k+1}=n$). We also put $e_0=k+1$. Let $\alpha_j=\vert\{i\in\{1,2,\ldots,k\}\colon e_i=j\}\vert$ denote the number of canonical northeast endpoints $(b_i^*,\pi_{b_i^*})$ such that $\pi_{b_i^*}$ is in the $j^{\text{th}}$ ascending run of $\pi$. For example, $\alpha_1=0$ because if $(b_i^*,\pi_{b_i}^*)$ is any canonical northeast endpoint of $\pi$, then we must have $b_i^*>d_1$. 

It is possible to show that a composition $(q_0,q_1,\ldots,q_k)$ of $n-k$ into $k+1$ parts is a valid composition of $\pi$ if and only if the following two conditions hold:
\begin{enumerate}[(a)]
\item For any $m\in\{0,1,\ldots,k\}$, \[\sum_{j=m}^{e_m-1}q_j\geq\sum_{j=m}^{e_m-1}\mu_j.\]
\item If $m,p\in\{0,1,\ldots,k\}$ are such that $m\leq p\leq e_m-2$, then \[\sum_{j=m}^pq_j\geq d_{p+1}-d_m-\sum_{j=m+1}^{p+1}\alpha_j.\]
\end{enumerate}
However, the proof of this fact is quite technical. Because we will not end up needing this result in any subsequent arguments, we omit its proof. 

We now have all the ingredients necessary for our main result. Let $F(\pi)$ denote the number of preimages of $\pi$ under $s$, and let $F(\pi,m)$ denote the number of such preimages with exactly $m$ descents. Moreover, recall that $C_i$ and $N(i,j)$ denote Catalan and Narayana numbers, respectively. 
\begin{theorem}\label{Thm3}
Let $\pi\in S_n$ be a permutation with exactly $k$ descents. Let $V(\pi)$ be as in Definition \ref{Def1}. 
We have \[F(\pi)=\sum_{(q_0,q_1,\ldots,q_k)\in V(\pi)}\prod_{i=0}^kC_{q_i}.\] For any nonnegative integer $m$, \[F(\pi,m)=\sum_{\substack{(q_0,q_1,\ldots,q_k)\in V(\pi)\\j_0+j_1+\cdots+j_k=m-k}}\prod_{t=0}^kN(q_t,j_t+1).\] Furthermore, the number of preimages of $\pi$ under $s$ with exactly $m$ valleys is given by 
\[2^{n-2m+1}\sum_{\substack{(q_0,q_1,\ldots,q_k)\in V(\pi)\\j_0+j_1+\cdots+j_k=m}}\prod_{t=0}^k {q_t-1\choose 2j_t-2}C_{j_t-1}.\] In these sums, $j_0,j_1,\ldots,j_k$ are nonnegative integers.  
\end{theorem}
\begin{proof} 
Lemma \ref{Lem1} and Definition \ref{Def1} tell us that the function $\varphi$ defined in \eqref{Eq12} provides a bijection between $\mathcal H_{\{0,2\}}(\pi)$ and $V(\pi)$. The result is then immediate from Lemmas \ref{Lem15}, \ref{Lem16}, and \ref{Lem17}. 
\end{proof}

\begin{corollary}\label{Cor1}
If $\pi\in S_n$ is a permutation with exactly $k$ descents and $m$ is a nonnegative integer, then \[F(\pi)\leq\sum_{i_0+i_1+\cdots+i_k=n-k}\prod_{t=0}^kC_{i_t}=\frac{2k+2}{n+1}{2n-2k-1\choose n}\] and \[F(\pi,m)\leq\sum_{\substack{i_0+i_1+\cdots+i_k=n-k\\j_0+j_1+\cdots+j_k=m-k}}\prod_{t=0}^kN(i_t,j_t+1),\] where the sums range over compositions $(i_0,i_1,\ldots,i_k)$ of $n-k$ and the second sum also ranges over weak compositions $(j_0,j_1,\ldots,j_k)$ of $m-k$. 
\end{corollary}
\begin{proof}
The inequalities follow immediately from Theorem \ref{Thm3}. Let $C(x)=\dfrac{1-\sqrt{1-4x}}{2x}$. It is straightforward to show that \[\sum_{n\geq 0}\left[\sum_{i_0+i_1+\cdots+i_k=n-k}\prod_{t=0}^kC_{i_t}\right]x^n=x^{2k+1}C(x)^{2k+2}.\] Catalan proved \cite{Catalan87} that the coefficient of $x^m$ in $C(x)^r$ is $[x^m]C(x)^r=\displaystyle{\frac{r}{m+r}{2m+r-1\choose m}}$, so it follows that \[\sum_{i_0+i_1+\cdots+i_k=n-k}\prod_{t=0}^kC_{i_t}=[x^{n-2k-1}]C(x)^{2k+2}=\frac{2k+2}{n+1}{2n-2k-1\choose n}.\]
\end{proof}

We end this section with a remark about a certain symmetry in the numbers $F(\pi,m)$. Fix a valid composition $(q_0,q_1,\ldots,q_k)$. If we let $j_t'=q_t-j_t-1$, then we find that \[\sum_{\substack{j_0+j_1+\cdots+j_t\\ =m-k}}\prod_{t=0}^kN(q_t,j_t+1)=\sum_{\substack{j_0'+j_1'+\cdots+j_t'\\=(n-m-1)-k}}\prod_{t=0}^kN(q_t,j_t+1)=\sum_{\substack{j_0'+j_1'+\cdots+j_t'\\=(n-m-1)-k}}\prod_{t=0}^kN(q_t,j_t'+1).\] It follows from Theorem \ref{Thm3} that $F(\pi,m)=F(\pi,n-m-1)$ for any integer $m$. This is a result originally due to Mikl\'os B\'ona \cite{Bona02}. 

\section{$3$-Stack \& $4$-Stack Sortable Permutations}
As mentioned in the introduction, for any integer $t\geq 3$, the best known upper bound for the number $\displaystyle{\lim_{n\to\infty}}\sqrt[n]{W_t(n)}$ is $(t+1)^2$. We remark that it is known that this limit actually exists because a sum of $t$-stack sortable permutations is $t$-stack sortable. Our goal in this section is to use Corollary \ref{Cor1} to improve this bound for $t=3$ and $t=4$.  Before we do so, we will need to take a short detour to state a lemma regarding generalized Narayana numbers, which we define by \[N_k(n,r)=\frac{k+1}{n}{n\choose r+k}{n\choose r-1}.\] Note that $N_0(n,r)=N(n,r)$. 

The proof of the following lemma, although perhaps combinatorially interesting in its own right, does not have any immediate relation to permutations or stack sorting. For this reason, the reader may bypass it without fear of missing any details that are crucial to our development.

Consider the set of lattice paths that use the steps $(1,0)$, $(0,1)$, $(-1,0)$, and $(0,-1)$ and never pass below the $x$-axis. Let $L_p(u,v)$ be the set of all such paths that start at $(0,0)$, end at $(u,v)$, and use exactly $p$ steps. It is known \cite{Guy00} that \begin{equation}\label{Eq16}
N_k(n,r)=\vert L_{n-1}(2r-n+k-1,k)\vert.
\end{equation}

\begin{lemma}\label{Lem11}
Let \[E(n,k,\ell)=\sum_{\substack{i_0+i_1+\cdots+i_k=n\\j_0+j_1+\cdots+j_k=\ell}}\prod_{t=0}^kN(i_t,j_t+1),\] where the sum ranges over all compositions $(i_0,i_1,\ldots,i_k)$ of $n$ and all weak compositions $(j_0,j_1,\ldots,$ $j_k)$ of $\ell$. We have \[E(n,k,\ell)=N_k(n,\ell+1).\] 
\end{lemma}
\begin{proof}
Choose a composition $(i_0,i_1,\ldots,i_k)$ of $n$ and a weak composition $(j_0,j_1,\ldots,j_k)$ of $\ell$. For each $t\in\{0,1,\ldots,k\}$, form a lattice path $\mathscr P_t$ in $L_{i_t-1}(2j_t-i_t+1,0)$. Equation \eqref{Eq16} tells us that the number of ways to choose each path $\mathscr P_t$ is $N(i_t,j_t+1)$, so the number of ways to make all of these choices is $E(n,k,\ell)$. 

Set $\mathscr P_0'=\mathscr P_0$. For each $t\in\{1,2,\ldots,k\}$, translate the path $\mathscr P_t$ to a new path $\mathscr P_t'$ so that the starting point of $\mathscr P_t'$ is one unit above the endpoint of $\mathscr P_{t-1}'$. Attach the endpoint of $\mathscr P_{t-1}'$ to the starting point of $\mathscr P_t'$ with a $(0,1)$ step. This results in a lattice path $P$ in $L_{n-1}(2\ell-n+k+1,k)$. The endpoint of $\mathscr P_t'$ is the last point in $P$ whose $y$-coordinate is $t$. Therefore, if we are given any path in $L_{n-1}(2\ell-n+k+1,k)$, we can determine exactly which composition $(i_0,i_1,\ldots,i_k)$, weak composition $(j_0,j_1,\ldots,j_k)$, and paths $\mathscr P_0,\mathscr P_1,\ldots,\mathscr P_k$ were used to obtain it in this fashion. This shows that \[E(n,k,\ell)=\vert L_{n-1}(2\ell-n+k+1,k)\vert=N_k(n,\ell+1).\]
\end{proof}

We can now use Lemma \ref{Lem11} in conjunction with Corollary \ref{Cor1} to see that if $\pi$ is a permutation of length $n$ with $k$ descents and $m$ is a nonnegative integer, then 
\begin{equation}\label{Eq17}
F(\pi,m)\leq N_k(n-k,m-k+1)=\frac{k+1}{n-k}{n-k\choose m+1}{n-k\choose m-k}.
\end{equation} We arrive at the following recursive upper bounds for the numbers $W_t(n)$ and $W_t(n,m)$. 
\begin{theorem}\label{Thm5}
Let $n\geq 2$ and $t\geq 1$ be integers, and let $m$ be a nonnegative integer. We have \[W_{t+1}(n)\leq\sum_{k=0}^{\left\lfloor (n-1)/2\right\rfloor}\frac{2k+2}{n+1}{2n-2k-1\choose n}W_t(n-1,k)\] and \[W_{t+1}(n,m)\leq\sum_{k=0}^m\frac{k+1}{n-k}{n-k\choose m+1}{n-k\choose m-k}W_t(n-1,k).\]
\end{theorem}
\begin{proof}
Let $Y_t(n,k)$ denote the set of $t$-stack sortable permutations in $S_n$ that have exactly $k$ descents and that have last entry $n$. A permutation $\sigma_1\sigma_2\cdots\sigma_{n-1}n$ is in $Y_t(n,k)$ if and only if $\sigma_1\sigma_2\cdots\sigma_{n-1}$ is a $t$-stack sortable permutation in $S_{n-1}$ with exactly $k$ descents. Therefore, $\vert Y_t(n,k)\vert=W_t(n-1,k)$. Recall that a permutation is said to be sorted if it is in the image of $s$. Observe that the last entry of any sorted permutation in $S_n$ must be $n$. Consequently, 
\begin{equation}\label{Eq22}
W_{t+1}(n)=\sum_{k=0}^{n-1}\sum_{\sigma\in Y_t(n,k)}F(\sigma)
\end{equation}
and 
\begin{equation}\label{Eq23}
W_{t+1}(n,m)=\sum_{k=0}^{n-1}\sum_{\sigma\in Y_t(n,k)}F(\sigma,m).
\end{equation}
It follows from \eqref{Eq22} and Corollary \ref{Cor1} that 
\[W_{t+1}(n)\leq\sum_{k=0}^{n-1}\frac{2k+2}{n+1}{2n-2k-1\choose n}\vert Y_t(n,k)\vert=\sum_{k=0}^{\left\lfloor (n-1)/2\right\rfloor}\frac{2k+2}{n+1}{2n-2k-1\choose n}W_t(n-1,k),\] where $k$ ranges from $0$ to $\left\lfloor (n-1)/2\right\rfloor$ in the second sum because $\displaystyle{{2n-2k-1\choose n}}=0$ when $k\geq n/2$. It follows from \eqref{Eq17} and \eqref{Eq23} that \[W_{t+1}(n,m)\leq\sum_{k=0}^{n-1}\frac{k+1}{n-k}{n-k\choose m+1}{n-k\choose m-k}\vert Y_t(n,k)\vert=\sum_{k=0}^m\frac{k+1}{n-k}{n-k\choose m+1}{n-k\choose m-k}W_t(n-1,k),\] where $k$ ranges from $0$ to $m$ in the second sum because $\displaystyle{n-k\choose m-k}=0$ when $k>m$. 
\end{proof}
We are now able to prove our upper bound for $\displaystyle{\lim_{n\to\infty}}\sqrt[n]{W_3(n)}$. We will make use of the formula 
\begin{equation}\label{Eq18}
W_2(n,k)=\frac{1}{(k+1)(2k+1)}{2n-k-1\choose k}{n+k\choose n-k},
\end{equation}
which appears as Corollary 9 in \cite{Dulucq}. 

\begin{theorem}\label{Thm6}
Let $\omega\approx 0.28839$ be the unique real root of the polynomial $4x^3-3x^2+4x-1$. We have \[\lim_{n\to\infty}\sqrt[n]{W_3(n)}\leq\frac{(2-\omega)^{2-\omega}(1+\omega)^{1+\omega}}{\omega^\omega(1-\omega)^{1-\omega}(2\omega)^{2\omega}(1-2\omega)^{1-2\omega}}\approx 12.53296.\]
\end{theorem}
\begin{proof}
Fix an integer $n\geq 4$, and let $K$ be an integer in the set $\{0,1,\ldots,\left\lfloor(n-1)/2\right\rfloor\}$ that maximizes \[\frac{2}{(n+1)(2K+1)}{2n-2K-1\choose n}{2n-K-3\choose K}{n+K-1\choose n-K-1}.\] Let $y=K/n$. Combining \eqref{Eq18} with the first inequality in Theorem \ref{Thm5} (with $t=2$) yields \[W_3(n)\leq\sum_{k=0}^{\left\lfloor(n-1)/2\right\rfloor}\frac{2}{(n+1)(2k+1)}{2n-2k-1\choose n}{2n-k-3\choose k}{n+k-1\choose n-k-1}\] 
\begin{equation}\label{Eq19}
\leq \frac{1}{(2K+1)}{2n-2K-1\choose n}{2n-K-3\choose K}{n+K-1\choose n-K-1}<n^2\frac{(2n-K-1)!(n+K-1)!}{n!K!(n-K)!(2K)!(n-2K)!}.
\end{equation} With the convention $0^0=1$, the elementary bounds
\begin{equation}\label{Eq20}
r^re^{-r}\leq r!\leq (r+1)^{r+1}e^{-r},
\end{equation} hold for all nonnegative integers $r$. Utilizing these bounds, we can deduce from \eqref{Eq19} that \[W_3(n)\leq n^2\frac{e^2(2n-K)^{2n-K}(n+K)^{n+K}}{n^nK^K(n-K)^{n-K}(2K)^{2K}(n-2K)^{n-2K}}=e^2n^2\left[\frac{(2-y)^{2-y}(1+y)^{1+y}}{y^y(1-y)^{1-y}(2y)^{2y}(1-2y)^{1-2y}}\right]^n.\] The function $f\colon[0,1/2)\to\mathbb R$ given by \[f(x)=\frac{(2-x)^{2-x}(1+x)^{1+x}}{x^x(1-x)^{1-x}(2x)^{2x}(1-2x)^{1-2x}}\] (again, we use the convention $0^0=1$) attains its maximum at $x=\omega$, so \[\sqrt[n]{W_3(n)}\leq\sqrt[n]{e^{2}n^2}f(\omega).\] 
\end{proof}

The derivation of our upper bound for $\displaystyle{\lim_{n\to\infty}}\sqrt[n]{W_4(n)}$ is a bit more involved than that of \\ $\displaystyle{\lim_{n\to\infty}}\sqrt[n]{W_3(n)}$, so we will require the following lemmas. The first provides an upper bound on the numbers $W_3(n,m)$, which we will need in order to use the first inequality in Theorem \ref{Thm5} with $t=3$.
 
\begin{lemma}\label{Lem12}
For $z\in\mathbb R$, let \[p_1(z)=-7+30z-24z^2-14z^3+12z^4-6z^5+2z^6\] and \[p_2(z)=81-324z+1188z^2-1404z^3-216z^4+1404z^5-972z^6
+432z^7-108z^8.\] Let \[\mathcal Q(z)=\frac13(z^2-z+1)-\frac{\sqrt[3]{2}(3-(z^2-z+1)^2)}{3\sqrt[3]{p_1(z)+\sqrt{p_2(z)}}}+\frac{\sqrt[3]{p_1(z)+\sqrt{p_2(z)}}}{3\sqrt[3]{2}}.\] For $0\leq u\leq v\leq 1/2$, let \[\xi(u,v)=\frac{(2-u)^{2-u}(1+u)^{1+u}}{4u^{3u}v^v(1-v-u)^{1-v-u}(v-u)^{v-u}(1-v)^{1-v}(1-u)^{1-u}}.\] Let $m\geq 0$ and $n\geq 1$ be integers with $m\leq n/2$, and put $y=m/n$. We have \[W_3(n,m)<e^4n^3\mathcal \xi(\mathcal Q(y),y)^n.\]
\end{lemma}

\begin{proof}
If $m=0$, then $y=0$ and $\xi(\mathcal Q(y),y)=\xi(0,0)=1$, so the desired result holds. Therefore, assume $m>0$. 

Let $K$ be the integer in the set $\{0,1,\ldots,m\}$ that maximizes \[\frac{1}{(n-K)(2K+1)}{n-K\choose m+1}{n-K\choose m-K}{2n-K-1\choose K}{n+K\choose n-K},\] and let $x=K/n$. Using \eqref{Eq18} and the second inequality in Theorem \ref{Thm5}, we find that \[W_3(n,m)\leq\sum_{k=0}^m\frac{1}{(n-k)(2k+1)}{n-k\choose m+1}{n-k\choose m-k}{2n-k-3\choose k}{n+k-1\choose n-k-1}\] \[\leq\frac{m+1}{(n-K)(2K+1)}{n-K\choose m+1}{n-K\choose m-K}{2n-K-3\choose K}{n+K-1\choose n-K-1}\] \begin{equation}\label{Eq21}
<n^3\frac{(n-K-1)!(2n-K-1)!(n+K-1)!}{m!(n-m-K)!(m-K)!(n-m)!K!(2n-2K)!(2K)!}.
\end{equation} The inequalities in \eqref{Eq20} now allow us to deduce from \eqref{Eq21} that \[W_3(n,m)<e^4n^3\frac{(n-K)^{n-K}(2n-K)^{2n-K}(n+K)^{n+K}}{m^{m}(n-m-K)^{n-m-K}(m-K)^{m-K}(n-m)^{n-m}K^{K}(2n-2K)^{2n-2K}(2K)^{2K}}\] \[=\frac{e^4n^3(2n-K)^{2n-K}(n+K)^{n+K}}{4^nK^{3K}m^m(n-m-K)^{n-m-K}(m-K)^{m-K}(n-m)^{n-m}(n-K)^{n-K}}=e^4n^3\xi(x,y)^n.\] 

We will show that $\xi(x,y)\leq\xi(\mathcal Q(y),y)$. It is straightforward to show that \[\frac{\partial}{\partial u}\log(\xi(u,y))=\log\left(\frac{(1-u^2)(1-y-u)(y-u)}{u^3(2-u)}\right)\] when $0<u<y$. Therefore, $\xi(u,y)$ is increasing (respectively, decreasing) as a function of $u$ if and only if the polynomial $p_y(u)=(1-u^2)(1-y-u)(y-u)-u^3(2-u)$ is positive (respectively, negative). We have \[p_y(u)=-u^3+(y^2-y+1)u^2-u-y^2+y.\] Using the fact that $0<y\leq 1/2$, one may easily verify that $p_y'(u)<0$ for all real $u$. Consequently, $p_y(u)$ has exactly one real root $c_y$. Because $p_y(y)<0<p_y(0)$, we must have $0<c_y<y$. Furthermore, the maximum value of $\xi(u,y)$ for $0<u<y$ occurs when $u=c_y$. Mathematica says that $c_y=\mathcal Q(y)$, so the lemma follows.  
\end{proof}

\begin{lemma}\label{Lem14}
For $v\in(0,1/2)$, let $\mathcal Q(v)$ be as in the preceding lemma. If $v\in(0.35,1/2)$, then $\mathcal Q(v)<4v/5$. If $v\in[0.22,0.35]$, then $\mathcal Q(v)<9v/10$. 
\end{lemma}
\begin{proof}
Recall that $\mathcal Q(v)$ is the unique real root of the strictly decreasing polynomial \[p_v(u)=-u^3+(v^2-v+1)u^2-u-v^2+v.\] If $v\in(0.35,1/2)$, then one may easily verify that $p_v(4v/5)<0$.  Similarly, one may show that $p_v(9v/10)<0$ if $v\in[0.22,0.35]$. 
\end{proof}

\begin{lemma}\label{Lem13}
Let $\mathcal Q$ and $\xi$ be as in Lemma \ref{Lem12}. Define a function $h\colon[0,1/2)\to\mathbb R$ by \[h(v)=\xi(\mathcal Q(v),v)\frac{(2-2v)^{2-2v}}{(1-2v)^{1-2v}}.\] We have $h(v)<21.97225$ for all $v\in[0,1/2)$.  
\end{lemma} 
\begin{proof}
One may show that \[\frac{d}{dv}\log(h(v))=\log\left(\frac{(1-v-\mathcal Q(v))(1-2v)^2}{4v(1-v)(v-\mathcal Q(v))}\right)+\mathcal Q'(v)\log\left(\frac{(1-\mathcal Q(v)^2)(1-v-\mathcal Q(v))(v-\mathcal Q(v))}{\mathcal Q(v)^3(2-\mathcal Q(v))}\right).\] Recall from the preceding proof that $\mathcal Q(v)$ is a root of the polynomial \[p_v(u)=(1-u^2)(1-v-u)(v-u)-u^3(2-u).\] This implies that \[\log\left(\frac{(1-\mathcal Q(v)^2)(1-v-\mathcal Q(v))(v-\mathcal Q(v))}{\mathcal Q(v)^3(2-\mathcal Q(v))}\right)=0,\] so \[\frac{d}{dv}\log(h(v))=\log\left(\frac{(1-v-\mathcal Q(v))(1-2v)^2}{4v(1-v)(v-\mathcal Q(v))}\right).\]

If $v\in(0,0.22)$, then \[\log\left(\frac{(1-v-\mathcal Q(v))(1-2v)^2}{4v(1-v)(v-\mathcal Q(v))}\right)\geq\log\left(\frac{(1-2v)^3}{4v^2(1-v)}\right)>0\] because $0\leq \mathcal Q(v)\leq v$. If $v\in(0.35,1/2)$, then we may use Lemma \ref{Lem14} to see that \[\log\left(\frac{(1-v-\mathcal Q(v))(1-2v)^2}{4v(1-v)(v-\mathcal Q(v))}\right)\leq\log\left(\frac{(1-2v)^2}{4v(v-4v/5)}\right)<0.\] This shows that $\log(h(v))$ is increasing for $v\in(0,0.22)$ and decreasing for $v\in(0.35,1/2)$, so the maximum value of $\log(h(v))$ occurs when $v\in[0.22,0.35]$. 

For $v\in[0.22,0.35]$, we may use Lemma \ref{Lem14} to find that \[\frac{d}{dv}\log(h(v))=\log\left(\frac{(1-v-\mathcal Q(v))(1-2v)^2}{4v(1-v)(v-\mathcal Q(v))}\right)\leq\log\left(\frac{(1-2v)^2}{4v(v-9v/10)}\right)<3.\] Consider the set $\mathcal A=\{0.22+0.13(n/10^4)\colon n\in\{0,1,\ldots,10^4\}\}$. Mathematica says that the largest value of $\log(h(v))$ for all $v$ in the finite set $\mathcal A$ is $3.0894788...$. Therefore, \[\max_{v\in[0.22,0.35]}\log(h(v))<\left(\max_{v\in[0.22,0.35]}\frac{d}{dv}\log(h(v))\right)(1/10^4)+3.08948<3.08978.\] As a consequence, we have $h(v)<e^{3.08978}<21.97225$ for all $v\in(0,1/2)$.
\end{proof}

\begin{theorem}\label{Thm7}
We have \[\lim_{n\to\infty}\sqrt[n]{W_4(n)}\leq 21.97225.\]
\end{theorem}
\begin{proof}
Fix an integer $n\geq 6$, and let $m$ be an integer in the set $\{0,1,\ldots,\left\lfloor(n-1)/2\right\rfloor\}$ that maximizes \[\frac{2m+2}{n+1}{2n-2m-1\choose n}W_3(n-1,m).\] Let $y=m/n$. Lemma \ref{Lem12} and the first inequality in Theorem \ref{Thm5} (with $t=3$) tell us that \[W_4(n)\leq\sum_{k=0}^{\left\lfloor(n-1)/2\right\rfloor}\frac{2k+2}{n+1}{2n-2k-1\choose n}W_3(n-1,k)\leq(m+1){2n-2m-1\choose n}e^4(n-1)^3\xi(Q(y),y)^{n-1}\] \[=e^4(n-1)^3(n-2m)(m+1)\frac{(2n-2m-1)!}{n!(n-2m)!}\xi(Q(y),y)^{n-1}.\] Now, \[\frac{(2n-2m-1)!}{n!(n-2m)!}\leq e\frac{(2n-2m)^{2n-2m}}{n^n(n-2m)^{n-2m}}=e\left[\frac{(2-y)^{2-y}}{(1-2y)^{1-2y}}\right]^n\] by \eqref{Eq20}. It follows that \[W_4(n)\leq e^5(n-1)^3(n-2m)(m+1)\xi(\mathcal Q(y),y)^{n-1}\left[\frac{(2-y)^{2-y}}{(1-2y)^{1-2y}}\right]^n\] \[=\frac{e^5}{\xi(\mathcal Q(y),y)}(n-1)^3(n-2m)(m+1)h(y)^n,\] where $h$ is the function defined in Lemma \ref{Lem13}. Lemma \ref{Lem13} now implies the theorem because $y\in[0,1/2)$.
\end{proof}

\section{Concluding Remarks}
We have attempted to minimize the computations involved in the proofs of the lemmas preluding Theorem \ref{Thm7}. Nevertheless, it would be interesting to obtain a less computational method for proving this theorem. We acknowledge that it is likely possible (yet probably cumbersome) to use the same techniques exploited in the preceding section to obtain upper bounds for $\displaystyle{\lim_{n\to\infty}}\sqrt[n]{W_t(n)}$ for relatively small values of $t\geq 5$. 

For a composition $q=(q_0,q_1,\ldots,q_k)\in \text{Comp}_{k+1}(n-k)$, let $\mathcal M(q)$ denote the set of permutations $\pi$ such that $q\in V(\pi)$. Let $\mathcal W_t(n)$ denote the set of $t$-stack sortable permutations of length $n$, and let $M_t(q)$ be the number of $t$-stack sortable permutations $\pi\in\mathcal M(q)$. It follows from Theorem \ref{Thm3} that \[W_{t+1}(n)=\sum_{\pi\in\mathcal W_t(n)}F(\pi)=\sum_{\pi\in\mathcal W_t(n)}\sum_{\substack{(q_0,q_1,\ldots,q_k)\\ \in V(\pi)}}\prod_{i=0}^k C_{q_i}=\sum_{k=0}^{n-1}\sum_{\substack{q=(q_0,q_1,\ldots,q_k)\\ \in \text{Comp}_{k+1}(n-k)}}M_t(q)\prod_{i=0}^kC_{q_i}.\] For this reason, the evaluation of $M_t(q)$ for various compositions $q$ could prove instrumental in the search for explicit formulas for $W_{t+1}(n)$ when $t\geq 2$. In particular, any significant results concerning values of $M_2(q)$, the number of $2$-stack sortable permutations $\pi$ for which $q\in V(\pi)$, would be of great interest.  

\section{Acknowledgments}
The work was funded by the University Scholars Program at the University of Florida. The author would like to express his deepest gratitude to Mikl\'os B\'ona for introducing the author to the stack-sorting algorithm and providing several useful comments toward the improvement of this manuscript's presentation. 

\end{document}